\input ./style/arxiv-general.cfg
\documentclass[aop,MSNbibl,dvips]{arximspdf}
\makeatletter
   \@ifpackageloaded{graphicx}{}{\usepackage{graphicx}}
\makeatother
\usepackage{mathrsfs}

\doi{10.1214/14-AOP958} 
\volume{43}
\issue{6}
\pubyear{2015}
\firstpage{3216}
\lastpage{3238}
\docsubty{FLA}

\makeatletter
\renewcommand{\emptyset}{\varnothing}
\newtheorem{lemma}[thm]{Lemma}
\newtheorem{cor}[thm]{Corollary}
\newtheorem{prop}[thm]{Proposition}
\newproclaim{rem}[thm]{Remark}
\makeatother

\begin{document}
\begin{frontmatter}

\title{Planar lattices do not recover from forest fires}
\runtitle{Planar lattices do not recover from forest fires}

\begin{aug}
\author[A]{\fnms{Demeter}~\snm{Kiss}\thanksref{T1}\ead[label=e1]{d.kiss@statslab.cam.ac.uk}},
\author[B]{\fnms{Ioan}~\snm{Manolescu}\corref{}\thanksref{T2}\ead[label=e2]{ioan.manolescu@unige.ch}}
\and
\author[C]{\fnms{Vladas}~\snm{Sidoravicius}\ead[label=e3]{vladas@impa.br}}
\runauthor{D. Kiss, I. Manolescu and V. Sidoravicius}
\affiliation{University of Cambridge and AIMR Tohoku University,
Universit\'e de Gen\`eve and~IMPA}
\address[A]{D. Kiss\\
Statistical Laboratory\\
Centre for Mathematical Sciences\\
University of Cambridge\\
Wilberforce Road\\
Cambridge, CB3 0WB\\
United Kingdom\\
\printead{e1}} 
\address[B]{I. Manolescu\\
D\'epartement de Math\'ematiques\\
Universit\'e de Gen\`eve\\
2--4 rue du Li\`evre, Gen\`eve\\
Switzerland\\
\printead{e2}}
\address[C]{V. Sidoravicius\\
IMPA\\
Estrada Dona Castorina 110\\
22460-320, Rio de Janeiro\\
Brazil\\
\printead{e3}}
\end{aug}
\thankstext{T1}{Supported in part by NWO Grant  613.000.804.}
\thankstext{T2}{Supported by the ERC Grant
AG CONFRA as well as the FNS.}

\received{\smonth{1} \syear{2014}}
\revised{\smonth{7} \syear{2014}}

%
\begin{abstract}
Self-destructive percolation with parameters $p,\delta$ is obtained by
taking a~site percolation configuration with parameter $p$,
closing all sites belonging to infinite clusters,
then opening every closed site with probability $\delta$, independently
of the rest.
Call $\theta(p,\delta)$ the probability that the origin is in an infinite
cluster in the configuration thus obtained.

For two-dimensional lattices, we show the existence of $\delta> 0$ such
that, for any $p > p_c$, $\theta(p,\delta) = 0$.
This proves the conjecture of van den Berg and Brouwer [\textit{Random Structures Algorithms} \textbf{24} (2004) 480--501],
who introduced the model.
Our results combined with those of van den Berg and Brouwer
[\textit{Random Structures Algorithms} \textbf{24} (2004) 480--501]
imply the
nonexistence of the infinite parameter forest-fire model.
The methods herein apply to site and bond percolation on any
two-dimensional planar lattice with sufficient symmetry.
\end{abstract}

%
\end{frontmatter}

\section{Introduction}

Self-destructive percolation was introduced in 2004 by van den Berg and
Brouwer \cite{VdBB}.
It may be formulated for both bond and site percolation; we choose to
consider the latter.
Fix some infinite connected graph $G$.

For $\delta, p \geq0$ consider a regular site percolation configuration
with intensity $p$.
Close all sites contained in the, possibly  many, infinite clusters; we
say infinite clusters are ``burned.''
Finally, open every closed site in the above configuration with
probability $\delta$, independently of all previous choices.
Call $\mathbb{P}_{p,\delta}$ the measure governing the configuration
thus obtained
and $\theta(p,\delta)$ the $\mathbb{P}_{p,\delta}$-probability that
a~given site
(called the origin) is in an infinite cluster.
Formal and extended definitions will be given in Section~\ref{sec:model}.

Let $\delta_c(p) = \inf\{\delta\dvtx\theta(p,\delta) > 0\}$, and
let $p_c =
p_c(G)$ denote the critical point for regular site percolation.
For $p < p_c$ and $\delta\geq0$, $\mathbb{P}_{p,\delta}$
is just a regular percolation measure with parameter $p + (1-p)\delta$.
In particular $\delta_c(p) = \frac{p_c - p}{1 - p}$ when $p < p_c$.
Consequently, self-destructive percolation is only interesting for $p
\geq p_c$.
In their original paper \cite{VdBB}, van den Berg and Brouwer
conjectured that,
for planar lattices, $\delta_c$ is uniformly bounded away from $0$
when $p>p_c$.

The conjecture is somewhat surprising. Recall that on planar lattices
there is no infinite cluster at $p = p_c$.
Hence, when $p$ is only slightly larger than $p_c$, the infinite
percolation cluster is very thin,
and it may be expected that, after burning it, opening only few sites
suffices to obtain a new infinite cluster.

Recently Ahlberg, Sidoravicius and Tykesson \cite{AST} proved that,
for nonamenable graphs $G$, the conclusion of the conjecture is false,
that is,  $\delta_c(p) \to0$ as \mbox{$p \to p_c$}.
The same has been shown by Ahlberg et al.\vspace*{1pt} \cite{ADKS} for
high-dimensional lattices
(more precisely for bond percolation on $\mathbb{Z}^d$ with $d$ large enough).

In two dimensions it has been proved in \cite{VdBB}, Proposition 3.1,
that $\delta_c(p) >0$ for any given $p > p_c$.
This was later strengthened by van den Berg and de Lima \cite{VdBL} to
the linear lower bound $\delta_c(p) \geq(p-p_c)/p$,
but a bound which is nonzero and uniform in $p$ could not be obtained.
In the present paper we prove the afore-mentioned conjecture.
For illustration we will consider site percolation on the
two-dimensional lattice $\mathbb{Z}^2$;
see Section~\ref{sec:model} for precise definitions.
Fix $G = \mathbb{Z}^2$ from now on.

%
\begin{thm}\label{thm:main}
There exists $\delta> 0$ such that, for all $p > p_c$, $\theta
(p,\delta) = 0$.
\end{thm}

%
\begin{rem} \label{rem:gen_lattice}
Theorem~\ref{thm:main} also holds for site and bond self-destructive
percolation
on any planar graph which is invariant under a translation (by some $u
\in\mathbb{R}^2 \setminus\{0\}$),
a rotation [of an angle $\varphi\in(0, \pi)$]
and reflection with respect to some line.
These symmetry conditions are needed for the RSW result of
Proposition~\ref{prop:bxp}.
Indeed Proposition~\ref{prop:bxp} may be adapted to lattices with the
symmetry conditions above.

In particular, an analogue of Theorem~\ref{thm:main} also holds for
site percolation on the triangular lattice
and bond percolation on the square lattice.
%
\end{rem}

Let us discuss some implications of Theorem~\ref{thm:main}. %
Let $\delta_c$ be the limit of $\delta_c(p)$ as $p\searrow p_c$.
Theorem~\ref{thm:main} 
together with the results in \cite{VdBBV} shows that the function
$(p,\delta)\rightarrow\theta(p,\delta)$ is continuous on the set
$[0,1]^2\setminus\{p_c\}\times(0,\delta_c]$,
while it is discontinuous on $\{p_c\}\times(0,\delta_c]$.

Our result has important consequences for forest fires, a class of
model introduced in \cite{DS-ff}.
Intuitively, an infinite-parameter forest fire is a process indexed by
$t\geq0$ defined as follows.
At time $t=0$, all sites are closed. As $t$ increases, sites open
independently at rate $1$.
When an infinite cluster appears it is immediately burned (i.e., all
its sites are closed).
Then sites become open again at rate $1$, etcetera.

It is not clear whether such a model actually exists.
We show in Section~\ref{sec:proof_steps} that our results combined with
those in \cite{VdBB} imply that infinite-parameter forest fires cannot
be defined on two-dimensional lattices.

To avoid the problems of definition, one can investigate the
$N$-parameter forest fire models
with $N<\infty$. That is, we modify the dynamics above by burning
clusters as soon as
their ``size'' reaches $N$. Our results with those of \cite{vdBB-forest-fires} provide some insight
to the behavior of these processes. We find a behavior which is quite
different compared to that
of a mean field version of the forest fire model; cf. \cite{RT}. See
Section~\ref{sec:proof_steps} for a
more detailed discussion.

\subsection*{Organization of the paper}

In Section~\ref{sec:model} we introduce the formal definitions and
notation used throughout the paper.
Once the notation is in place,
in Section~\ref{sec:proof_steps} we state a result on certain
box-crossing probabilities (Theorem~\ref{thm:cross})
and show how Theorem~\ref{thm:main} can be deduced from it.
We also discuss in more detail its implications for forest fire models.
Theorem~\ref{thm:cross} is our main contribution.

Section~\ref{sec:arms} contains a review of the notion of arm events
essential, to the proofs of the next section.
In Section~\ref{sec:cross} we provide a delicate counting argument
which proves Theorem~\ref{thm:cross}.

\section{Definitions and notation}\label{sec:model}

\subsection{The model}

Let $\mathbb{Z}^2$ denote the square lattice with vertices $V(\mathbb
{Z}^2)$ (also
called \emph{sites}) and edges $E(\mathbb{Z}^2)$.
For sites $x,y \in V(\mathbb{Z}^2)$ we write $x \sim y$, alternatively $(x,y)
\in E(\mathbb{Z}^2)$, when $\Vert x-y\Vert_2 = 1$.
Set $\Omega= \{0,1\}^{V(\mathbb{Z}^2)}$.
We call an element $\omega\in\Omega$ a \emph{configuration} and
write $\{\omega(x)\dvtx x \in V(\mathbb{Z}^2)\}$ for its coordinates.
A site $x$ with $\omega(x)=1$ is called \emph{open} (or $\omega
$-open when
the configuration needs to be specified),
while one with $\omega(x) = 0$ is called \emph{closed}.

A path on $\mathbb{Z}^2$ is a sequence of sites $\gamma= (u_0, \ldots
, u_n)$ with
$u_i \sim u_{i+1}$ for $i = 0, \ldots, n-1$.
Moreover we ask all paths to be self-avoiding, that is, for the
vertices $u_0, \ldots, u_n$ to be pairwise distinct.
A path is called $\omega$-open (resp., $\omega$-closed) for a~configuration
$\omega$ if all its vertices are $\omega$-open (resp., $\omega$-closed).

For a configuration $\omega$ and $x, y \in V(\mathbb{Z}^2)$, we say
$x$ is \emph{connected} to $y$ in $\omega$, and write $x \stackrel{\omega}{\leftrightarrow} y$,
if there exists an $\omega$-open path with endpoints $x$ and $y$.
We write $x \stackrel{\omega}{\leftrightarrow} \infty$ and say that $x$ is
\emph{connected to infinity}
if there exists an infinite $\omega$-open path starting at $x$.
Finally we write $x \stackrel{\omega}{\nleftrightarrow} y$ and $x
\stackrel{\omega}{\nleftrightarrow} \infty$ for the
negations of the above events.
A \emph{cluster} is a connected component of the graph induced by the
open sites of $\mathbb{Z}^2$.

For $p \in[0,1]$, let $\mathbb{P}_p$ be the site percolation measure
on $\mathbb{Z}
^2$ with intensity\vspace*{1pt} $p$.
That is, $\mathbb{P}_p$ is the product measure on $\Omega$ with
$\mathbb{P}_p(\omega(x) =1)
= p$ for all $x \in V(\mathbb{Z}^2)$.
Finally let $p_c = \sup\{p \geq0 \dvtx \mathbb{P}_p(0 \stackrel{\omega}{\leftrightarrow}\infty) = 0\}$.
For $p > p_c$ it is well known that there exists $\mathbb{P}_p$-a.s. a unique
infinite cluster.
For this and further details on percolation we direct the reader to
\cite{Grimmett_Percolation}.

Let $p \in[0,1]$, and consider a configuration $\omega$ chosen according
to $\mathbb{P}_p$.
We define a modification of $\omega$, called $\overline\omega$, as follows.
For $x \in V(\mathbb{Z}^2)$,
\[
\overline\omega(x) = \cases{ 1, & \quad$\mbox{if } \omega(x) =1 \mbox{ and } x
\stackrel{\omega} {\nleftrightarrow} \infty$,
\cr
0, &\quad$\mbox{otherwise}$.}
\]
Let $\delta\geq0$ and $\sigma$ be a configuration chosen according
to $\mathbb{P}
_\delta$, independently of $\omega$.
The enhancement of $\overline\omega$ with intensity $\delta$ is
$\overline\omega^\sigma(x) =
\overline\omega(x) \vee\sigma(x)$.

Let $\mathbb{P}_{p, \delta}$ denote the probability measure governing
$\omega$, $\sigma
$ and thus $\overline\omega$ and $\overline\omega^\sigma$.
To avoid confusion, when working with $\mathbb{P}_{p,\delta}$, we
will usually
state to which configuration we refer.
When writing simply $\mathbb{P}_{p,\delta}(A)$ we mean $\mathbb
{P}_{p,\delta}(\overline\omega^\sigma
\in A)$.
Let
\[
\theta(p,\delta) = \mathbb{P}_{p, \delta} \bigl( 0 \stackrel{\overline{\omega}^\sigma}{\longleftrightarrow} \infty \bigr).
\]
Note that $\mathbb{P}_{p, \delta}$ is increasing in $\delta$, hence
so is $\theta$.

\subsection{Further notation} \label{ssec:further_not}


Let $\operatorname{dist}(\cdot,\cdot)$ denote the $L^\infty$
distance on $\mathbb{Z}^2$. That is,
\[
\operatorname{dist}(x,y) = \max \bigl(\vert x_1-y_1\vert,
\vert x_2-y_2\vert \bigr) \qquad\mbox{for }
x=(x_1,x_2),y=(y_1,y_2)\in
\mathbb{Z}^2.
\]
For $u \in V(\mathbb{Z}^2)$ and $n \geq0$,
denote by $\Lambda_n(u)$ the ball of radius $n$ around $u$ for the
$L^\infty$ distance.
Hence $\Lambda_n(u) = ([-n,n]^2 + u )\cap\mathbb{Z}^2$.
When $u$ is omitted, it is assumed equal to the origin.
We will usually identify regions of the plane with the set of vertices
they contain.


For $A \subset V(\mathbb{Z}^2)$, we call the (\textit{outer}) \textit{boundary}
of $A$
the set
\[
\partial A = \bigl\{y \in V \bigl(\mathbb{Z}^2 \bigr) \setminus A
\dvtx y \sim x \mbox{ for some } x \in A \bigr\};
\]
the \emph{internal boundary} of $A$ is the set $\partial_iA =
\partial(A^c)$.
The diameter of the set $A$ is $\operatorname{diam}(A) = \sup\{
\operatorname{dist}(x,y)\dvtx x,y \in
A \}$.

For a configuration $\omega$ and $x,y \in A \subset V(\mathbb{Z}^2)$,
we say $x$
is $\omega$-connected to $y$ in~$A$, and write $x \stackrel{\omega, A}{\longleftrightarrow} y$,
if there exists an $\omega$-open path with endpoints $x$ and $y$, fully
contained in $A$.


The matching graph of $\mathbb{Z}^2$, written $(\mathbb{Z}^{2})^*$,
has the same vertex set as $\mathbb{Z}^2$ and an edge between any two vertices
of the same face of $\mathbb{Z}^2$.
We say that $x$ and $y$ are \mbox{${}^{*}$-\emph{connected}}, and write $x
\mathop{\stackrel{\omega}{\leftrightarrow}\!\!\!{}^*} y$,
if there exists $\omega$-closed path in $(\mathbb{Z}^2)^*$ with
endpoints $x$ and $y$.
The notion of matching graph is proper to site percolation, so
when working with bond percolation it should be replaced by the dual graph.
For more details on matching and dual graphs consult \cite{Grimmett_Percolation}.


For $m,n \in\mathbb{N}$, we define the rectangular box $B(m,n) = [0,m]
\times[0,n]$.
The sides of $B(m,n)$ are the sets $[0,m]\times\{0\}$, $[0,m]\times
\{n\}$, $\{0\}\times[0,n]$ and $\{m\}\times[0,n]$,
and they are called the bottom, top, left-hand and right-hand side,
respectively.
Given a configuration $\omega$, we say $B(m,n)$ is crossed horizontally
if there exists an $\omega$-open path $\gamma$
contained in $B(m,n)$, with one endpoint on the left-hand side and
one on the right-hand side of $B(m,n)$.
We say it is crossed vertically if an $\omega$-open path contained in $B
(m,n)$ connects the top and the bottom.
We write $\mathcal{C}_h(m,n)$ and $\mathcal{C}_v(m,n)$ for the events
that $B(m,n)$ is
crossed horizontally, respectively,   vertically.
If $R$ is a translate of the box $B(m,n)$, we write $\mathcal{C}_h(R)$
and $\mathcal{C}_v(R)$
for the appropriate translations of $\mathcal{C}_h(m,n)$ and $\mathcal
{C}_v(m,n)$.


Finally, we mention a well-known result for standard percolation that
is essential to our analysis.
This type of result was initially proved separately by Russo \cite{Russo} and
Seymour and Welsh \cite{Seymour-Welsh}, hence the name of
Russo--Seymour--Welsh (RSW) result.
For reference we direct the reader to \cite{Grimmett_Percolation}, Theorem~11.70.
Extensions to percolation models on general graphs with the symmetries
mentioned in Remark~\ref{rem:gen_lattice}
are discussed in detail in \cite{Kesten_book}, Section~6.

%
\begin{prop}[(RSW)]\label{prop:bxp}
There exists a constant $\alpha>0$ such that, for all $n\geq1$,
%
%
\begin{equation}
\label{eq:bxp} \mathbb{P}_{p_c} \bigl(\mathcal{C}_h(2n,n)
\bigr) \geq\alpha.
\end{equation}
The analogue holds for *-crossings on the matching graph.
\end{prop}
%

\section{Box-crossing estimates and consequences for forest
fires}\label
{sec:proof_steps}

\subsection{Crossing boxes after the burn}
The proof of Theorem~\ref{thm:main} is based on a~crossing-probability
estimate.
Some additional notation is needed.

Let $R_n = [-2n,2n] \times[0,n]$ and $S_n = [-3n,3n] \times[0,n]$.
For a configuration $\omega$ let $\chi$ be the set of sites $x\in S_n$
which are connected to both the left-hand and right-hand sides of $S_n$
by open paths contained in $S_n$.
Define a configuration $\widetilde\omega$ by setting, for $x \in S_n$,
%
%
\begin{equation}
\label{def:omega_tilde} \widetilde\omega(x) = \cases{0, & \quad $\mbox{if } x \in\chi\cup
\partial \chi$,
\cr
1, & \quad $\mbox{otherwise}$.}
\end{equation}
In other words, the $\omega$-open clusters containing horizontal crossings
of $S_n$
are declared closed in $\widetilde\omega$, as are their boundaries.
All other sites are opened.
The value of $\widetilde\omega$ outside of $S_n$ is irrelevant for
our purposes;
for concreteness we take $\widetilde\omega= 0$ there.
Finally we enhance the vertices inside $R_n$ by setting
%
%
\begin{equation}
\label{def:omega_tilde_delta} \widetilde\omega^\sigma(x) = \cases{\widetilde\omega(x)
\vee\sigma(x), & \quad $\mbox{if } x \in R_n$,
\cr
\widetilde\omega(x),
&\quad $\mbox{otherwise}$.}
\end{equation}

%
\begin{thm}\label{thm:cross}
There exist constants $\delta, \lambda, c >0$ such that, for all $n
\geq1$,
%
%
\begin{equation}
\label{eq:cross} \mathbb{P}_{p_c,\delta} \bigl[\omega\in\mathcal
{C}_h(S_n) \mbox{ and } \widetilde\omega^\sigma
\in\mathcal{C}_v(R_n) \bigr] \leq c n^{-\lambda}.
\end{equation}
\end{thm}

Similar statements to (\ref{eq:cross}) have been shown to imply
Theorem~\ref{thm:main}, but none has been proved.
See, for instance, \cite{VdBB}, Conjecture 3.2, and \cite{vdBB-forest-fires}, Conjecture~2.1.
Our criterion was inspired by the previous; the slightly different
formulation is particularly
adapted to our proof.

Theorem~\ref{thm:cross} will be proved in Section~\ref{sec:cross}.
For completeness we give a proof of Theorem~\ref{thm:main} from
Theorem~\ref{thm:cross} that follows the steps of \cite{VdBB}.
We start with a corollary which requires some additional notation.

Recall the definition of $\Lambda_n$ from Section~\ref{ssec:further_not}.
Consider some $n \in\mathbb{N}$, and define the annulus $\mathrm
{A}(n,2n) = \Lambda
_{2n} \setminus\Lambda_{n-1}$.
A circuit in $\mathrm{A}(n,2n)$ is a path contained in $\mathrm
{A}(n,2n)$ that
separates the origin from infinity.
For a configuration $\omega$, define a new modification $\check
{\omega}$ of
$\omega$,
by closing all sites that are connected by an $\omega$-open path in
$\mathrm{A}(n,2n)$
to an $\omega$-open circuit in $\mathrm{A}(n,2n)$.
As above, for a second configuration $\sigma$, set $\check{\omega
}^\sigma=
\check\omega\vee\sigma$.

%
\begin{cor} \label{cor:circ}
There exists a constant $\rho>0$ such that, with $\delta$ as in
Theorem~\ref{thm:cross},
\[
\mathbb{P}_{p_c,\delta} \bigl(\partial\Lambda_{n-1}\stackrel{\check{\omega}^\sigma}{\longleftrightarrow}\partial\Lambda_{2n} \bigr)
\leq1-\rho
\]
for all $n\geq1$.
\end{cor}

Before we dive in the proofs of Corollary~\ref{cor:circ} from
Theorem~\ref{thm:cross}
and of Theorem~\ref{thm:main} from Corollary~\ref{cor:circ},
let us turn to some other implications of
Theorem~\ref{thm:cross} and Corollary~\ref{cor:circ}.

\subsection{Consequences for forest fires}

The following was stated as a conditional result in \cite{VdBB}. Our
results imply it.
%
%
\begin{thm}[(Theorem~4.1 of \cite{VdBB})]\label{thm:no_ff}
The infinite-parameter forest fire process does not exist on $\mathbb{Z}^2$.
\end{thm}

In \cite{VdBB} the above was stated conditionally on \cite{VdBB}, Conjecture~3.2.
While the latter is not obviously implied by our results,
its main consequence, \cite{VdBB}, Lemma~3.4, is equivalent to
Corollary~\ref{cor:circ} above.
The proof of the theorem in \cite{VdBB} is based solely on \cite{VdBB}, Lemma 3.4.

The intuition behind Theorem~\ref{thm:no_ff} is the following.
Suppose an infinite-parameter forest fire process is defined, and let
$t_c$ (defined by $1 - e^{-t_c} = p_c$)
be the time when fires start to appear.
No fires ignite on $[0,t_c]$ since no infinite cluster is produced.
But for any $t > t_c$ at least one infinite cluster was produced and
burned before~$t$.
Thus an infinity of burning times have to accumulate after~$t_c$.
But Theorem~\ref{thm:main} suggests that there exists a universal
$\tau
> 0$ such that,
after one fire, the process needs at least time $\tau$ to recover and
recreate a new infinite cluster.
This leads to a contradiction, hence the nonexistence of the process.

In \cite{vdBB-forest-fires} van den Berg and Brouwer stated several results
for finite-parameter forest fires conditionally on \cite{vdBB-forest-fires}, Conjecture~2.1.
Our Theorem~\ref{thm:cross} implies this conjecture, and hence their results.
We will state two of them. In the following $\eta^{[N]}$ denotes the
$N$-parameter forest fire process.
We say $\eta^{[N]}$ has a fire in $\Lambda_m$ when a cluster intersecting
$\Lambda_m$ reaches size $N$ and is burned.
%
%
\begin{thm}[(Theorem 4.2 and Proposition 4.3 of \cite{vdBB-forest-fires})]
There exists $t>t_c$ such that for all $m \geq0$,
\begin{eqnarray*}
\liminf_{N\rightarrow\infty}\mathbb{P} \bigl(\eta^{[N]}\mbox{ has a
fire in } \Lambda_m \mbox{ before time } t \bigr) & \leq & 1/2,
\\
\lim_{N\rightarrow\infty}\mathbb{P} \bigl(\eta^{[N]}\mbox{ has at
least } 2 \mbox{ fires in } \Lambda_m \mbox{ before time } t \bigr)
&=& 0.
\end{eqnarray*}
\end{thm}

The interested reader is referred to \cite{VdBB} and \cite{vdBB-forest-fires}
for precise definitions of forest fires and more details.
We conclude the section with the proofs of Corollary~\ref{cor:circ}
and Theorem~\ref{thm:main} given Theorem~\ref{thm:cross}.

\begin{figure}[b]

\includegraphics{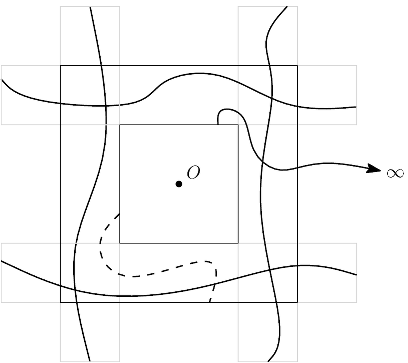}

\caption{A situation with $\omega\in C_n$ and $\partial\Lambda
_{n-1} \stackrel{\omega}{\leftrightarrow}
\infty$ (the bold paths are open in $\omega$)
and such that there exists an $\overline{\omega}^\sigma$-open
crossing of the
annulus $\mathrm{A}(n,2n)$ (dashed path).
Then $\overline{\omega}^\sigma\leq\check{\omega}^\sigma$, and
there exists an
$\check{\omega}^\sigma$-open crossing in the easy direction of one
of the
four rectangles forming $\mathrm{A}(n,2n)$.
Note that any site of a $\omega$-open horizontal crossing of a gray
rectangle is closed in $\check{\omega}$.}
\label{fig:annulus_cross}
\end{figure}

\subsection{Proofs}

Figure~\ref{fig:annulus_cross} sums up the proofs of both
Corollary~\ref
{cor:circ} and Theorem~\ref{thm:main}.

\begin{pf*}{Proof of Corollary~\ref{cor:circ} from
Theorem~\ref{thm:cross}}
For $n \geq1$ denote the four $6n \times n$ rectangles surrounding
$\Lambda_n$ by
\begin{eqnarray*}
S_B &=& [-3n,3n] \times[-2n,-n], \qquad S_T = [-3n,3n]
\times[n,2n],
\\
S_L &=& [-2n,-n] \times[-3n,3n], \qquad S_R = [n,2n]
\times[-3n,3n].
\end{eqnarray*}
Let $R_B = [-2n,2n] \times[-2n,-n]$, and define similarly $R_T,R_L$
and $R_R$.
Also let $C_n = \mathcal{C}_h(S_B) \cap\mathcal{C}_h(S_T)\cap
\mathcal{C}_v(S_L) \cap\mathcal{C}_v(S_R)$
and note that,
if $\omega\in C_n$, then $\omega$ contains an open circuit in
$\mathrm{A}(n,2n)$.
By Proposition~\ref{prop:bxp} and the FKG inequality for regular percolation,
there exists $\rho> 0$ such that $\mathbb{P}_{p_c}(C_n) \geq2\rho$
for all
$n \geq1$.

%

Fix $\delta$ as in Theorem~\ref{thm:cross}.
Let $\omega\in C_n$ and $\sigma$ be such that $\check{\omega
}^\sigma$ contains an
open path
$\gamma$ between $\partial\Lambda_{n-1}$ and $\partial_i\Lambda_{2n}$.
Then it is easy to see that $\gamma$ contains a crossing in the easy
direction of one of the rectangles $R_B,R_T,R_L$ and $R_R$.
In other words,
%
%
\begin{equation}
\label{eq:pf_circ_1}
\check{\omega}^\sigma\in\mathcal{C}_v(R_B)
\cup \mathcal{C}_v(R_T) \cup\mathcal{C}_h(R_L)
\cup\mathcal{C}_h(R_R).
\end{equation}
%
Suppose, for instance, that $\check{\omega}^\sigma\in\mathcal{C}_v(R_B)$.
Since $\omega\in C_n$, all sites connected to a horizontal crossings of
$S_B$ are closed in $\check{\omega}$,
and (\ref{eq:cross}) implies that
\[
\mathbb{P}_{p_c,\delta} \bigl[\omega\in C_n \mbox{ and
} \check{ \omega}^\sigma\in\mathcal{C}_v(R_B)
\bigr] \leq c n^{-\alpha}.\
\]
This combined with by (\ref{eq:pf_circ_1}), and the bound $\mathbb{P}
_{p_c}(C_n) \geq2\rho$ gives
\begin{eqnarray*}
&& \mathbb{P}_{p_c,\delta} \bigl(\partial\Lambda_{n-1} \stackrel{\check{\omega}^\sigma}{\longleftrightarrow}\partial\Lambda_{2n} \bigr)\\
&& \qquad \leq  \mathbb {P}_{p_c,\delta} \bigl(\omega\in C_n \mbox{ and }
\partial \Lambda _{n-1}\stackrel{\check{\omega}^\sigma}{\longleftrightarrow} \partial \Lambda_{2n} \bigr) + \mathbb
{P}_{p_c,\delta}(\omega\notin C_n)
\\
&&\qquad  \leq  4cn^{-\alpha} + 1 - 2\rho.
\end{eqnarray*}
Taking $n$ large enough in the above completes the proof of
Corollary~\ref{cor:circ}.
\end{pf*}

\begin{pf*}{Proof of Theorem~\ref{thm:main} from
Corollary~\ref{cor:circ}}
Corollary~\ref{cor:circ} gives crossing probability estimates for
measures $\mathbb{P}_{p_c,\delta}$.
We start by extending these to measures $\mathbb{P}_{p,\delta'}$,
with $p>p_c$.
Let $\delta> 0$ be given by Theorem~\ref{thm:cross}.
Fix some $p > p_c$, and let $\delta' >0$ be such that
$p + (1-p)\delta' \leq p_c + (1-p_c)\delta$.

It is easy to check (see, e.g., \cite{VdBB}, Corollary~2.4)
that the configuration $\check\omega^\sigma$ obtained from $\mathbb
{P}_{p_c, \delta}$
stochastically dominates
that obtained from $\mathbb{P}_{p, \delta'}$.
In particular,
\[
\mathbb{P}_{p_c,\delta} \bigl(\partial\Lambda_{n-1}\stackrel{\check{\omega}^\sigma}{\longleftrightarrow}\partial\Lambda_{2n} \bigr)
\geq\mathbb {P}_{p,\delta'} \bigl( \partial\Lambda_{n-1}\stackrel{\check{\omega}^\sigma}{\longleftrightarrow}\partial\Lambda_{2n}
\bigr).
\]
This together with Corollary~\ref{cor:circ} implies that
\[
\mathbb{P}_{p,\delta/2} \bigl(\partial\Lambda_{n-1}\stackrel{\check{\omega}^\sigma}{\longleftrightarrow}\partial\Lambda_{2n} \bigr)
\leq1 - \rho
\]
for all $p$ sufficiently close to $p_c$ and all sufficiently large $n$.
We claim that the above yields $\theta(p, \delta/2) = 0$.

Since $p > p_c$, the probability that $\partial\Lambda_{n-1}$ does
not have
an $\omega$-open path to $\infty$
is at most $\rho/2$ for $n$ sufficiently large.
Moreover, if $\partial\Lambda_{n-1} \stackrel{\omega}{\leftrightarrow} \infty
$, then $\overline{\omega} \leq
\check{\omega}$.
Hence, for all sufficiently large $n$,
\begin{eqnarray*}
\mathbb{P}_{p,\delta/2} \bigl(\partial\Lambda_{n-1} \stackrel{\overline{\omega}^\sigma}{\longleftrightarrow} \infty \bigr) &\leq &
\mathbb{P}_{p,\delta/2} \bigl(\partial\Lambda _{n-1}\stackrel{\check{\omega}^\sigma}{\longleftrightarrow} \Lambda_{2n} \bigr) +
\mathbb{P}_{p,\delta/2} \bigl(\partial \Lambda_{n-1}\stackrel{\omega} {
\nleftrightarrow} \infty\bigr)
\\
&\leq & 1 - \rho+ \rho/2.
\end{eqnarray*}
The event $\{\overline{\omega}^\sigma\mbox{ contains an infinite
cluster}\}$ is
translation invariant;
thus its probability is either $0$ or $1$. The above excludes the
latter, hence $\theta(p, \delta/2) = 0$.
\end{pf*}

\section{Arm events}\label{sec:arms}



A \emph{color sequence} of\vspace*{1pt} length $k$ is a sequence $\varsigma\in\{
0,1\}^k$.
Fix such a color sequence $\varsigma$, a vertex $u\in\mathbb{Z}^2$
and integers $n
\leq N$.
We write $\mathcal{A}_\varsigma(u;n,N)$ for the event that there
exist $k$ pairwise
disjoint paths
$\gamma_1, \ldots, \gamma_k$ such that, for $j = 1, \ldots, k$:
\begin{itemize}
\item if $\varsigma_j =1$, then $\gamma_j$ is a open path on $\mathbb
{Z}^2$, and if
$\varsigma_j = 0$, then $\gamma_j$ is a closed path on $(\mathbb{Z}^2)^*$;
\item$\gamma_j \subset\Lambda_N(u) \setminus\Lambda_n(u)$ and has one
endpoint in $\partial\Lambda_n(u)$ and the other in $\partial
_i\Lambda_N(u)$;
\item the endpoints of $\gamma_1, \ldots, \gamma_k$ are placed in
counter-clockwise order on $\partial\Lambda_n(u)$.
\end{itemize}
The paths $\gamma_j$ are called \textit{arms}; and the event
$\mathcal{A}_\varsigma
(u;n,N)$ is called an \emph{arm event}.
When $u$ is omitted, it is assumed to be the origin.
The probabilities of arm events are denoted by $\pi_\varsigma(n,N) =
\mathbb{P}
_{p_c} (\mathcal{A}_\varsigma(n,N))$.

For very small values of $n$, $\mathcal{A}_\varsigma(n,N)$ could be
empty because
of geometric constraints.
It will be convenient to redefine $\mathcal{A}_\varsigma(n,N)$ as
$\mathcal{A}_\varsigma(|\varsigma
|,N)$ when $n \leq|\varsigma|$.
Let $\mathcal{A}_\varsigma(n) = \mathcal{A}_\varsigma(0,n)$ and
$\pi_\varsigma(n) = \pi_\varsigma(0,n)$.


A related notion is that of \emph{half-plane arm events}.
Let $\mathbb{H}= \mathbb{R}\times[0,\infty)$ be the upper half-plane.
Define $\mathcal{A}_\varsigma^{hp}(n,N)$ as the event $\mathcal
{A}_\varsigma(n, N)$,
with the additional restriction that the arms $\gamma_1, \ldots,
\gamma_k$ are
all contained in $\mathbb{H}$
and that $\gamma_1$ is the right-most arm.

The notation for arm events extends to half-plane arm events, thus\break 
$\pi_\varsigma^{hp}(n,N) = \mathbb{P}_{p_c} (\mathcal{A}_\varsigma
^{hp}(n,N))$, $\mathcal{A}_\varsigma
^{hp}(n) = \mathcal{A}_\varsigma^{hp}(0,n)$ and $\pi_\varsigma
^{hp}(n) = \pi_\varsigma^{hp}(0,n)$.


Here are two well-known properties of arm events.
%

\begin{prop}\label{pro:Nolin}
Fix a color sequence $\varsigma$.
There exists a constant {$c = c(\varsigma) > 0$} such that, for $0
\leq n
\leq m \leq N$,
%
%
\begin{eqnarray}
\label{eq:quasimult} c \pi_\varsigma(n,m) \pi_\varsigma(m,N) &\leq &
\pi_\varsigma(n,N) \leq\pi_\varsigma(n,m) \pi_\varsigma(m,N),
\\
\label{eq:short_arm} \pi_\varsigma(n,2n) & \geq & c.
\end{eqnarray}
The above also holds for half-plane arm events.
\end{prop}

The proposition is not specific to site percolation on $\mathbb{Z}^2$;
the only thing needed for the proof is the crossing estimate (\ref{eq:bxp}).
The bound (\ref{eq:quasimult}) first appeared in \cite{Kesten87}, combination of Lemmas 4~and~6,
while (\ref{eq:short_arm}) is a simple consequence of (\ref{eq:bxp}).
For a modern treatment of Proposition~\ref{pro:Nolin}
and for other proofs in this section, we refer the reader to the survey
\cite{Nolin}.

We also need to introduce the notion of \emph{arms with defects}.
Let $\mathcal{A}_\varsigma^*(n,N)$ be the set of configurations
$\omega$ such that
there exists a point $u$
and a configuration $\omega'$ equal to $\omega$ outside $\Lambda
_3(u)$ with $\omega
' \in\mathcal{A}_\varsigma(n,N)$.
All the notation defined above extends to arm events with defects, with
the attached asterisk.

%
\begin{prop}[(Proposition 18 of \cite{Nolin})]\label{prop:arms_defects}
Fix a color sequence $\varsigma$.
There exists a positive constant $C = C(\varsigma)$ such that, for all $n
\leq N$,
%
%
\begin{eqnarray}
\mathbb{P}_{p_c} \bigl(\mathcal{A}_\varsigma^*(n,N) \bigr) & \leq
& C \bigl(1+ \log(N/n) \bigr) \pi_\varsigma(n,N)\quad \mbox{and}
\nonumber
\\[-8pt]
\label{eq:arm_defects}
\\[-8pt]
\nonumber
\mathbb{P}_{p_c} \bigl(\mathcal{A}_\varsigma^{hp *}(n,N)
\bigr) & \leq & C \bigl(1+ \log(N/n) \bigr) \pi_\varsigma^{hp}(n,N).
\end{eqnarray}
\end{prop}

%
\begin{rem}
In \cite{Nolin} an arm event with a defect is defined as a
modification of the event $\mathcal{A}_\varsigma(n,N)$
where the arms are allowed to have at most one vertex of the opposing color.
Our definition is slightly different; nevertheless, Nolin's proof
readily extends to our case.
\end{rem}


In the rest of the paper, the following types of arm events will play a
special role.
Call $\mathcal{A}_1$, $\mathcal{A}_5$ and $\mathcal{A}_6$ the event
$\mathcal{A}_\varsigma$ with
$\varsigma= (1)$, $\varsigma= (1,0,0,1,0)$ and $\varsigma=
(0,1,0,0,1,0)$, respectively.
In addition, write $\mathcal{A}_3^{hp}$ and $\mathcal{A}_4^{hp}$ for
the event $\mathcal{A}
^{hp}_\varsigma$
with $\varsigma= (1,0,1)$ and $\varsigma= (1,0,0,1)$, respectively.
The same notation applies to~$\pi$.

The following is a well-known consequence of (\ref{eq:bxp}).
See \cite{Nolin}, Theorem~24,  and~\cite{Grimmett_Percolation}, Theorem~11.89.
%

\begin{prop}
There exist constants $\lambda, c, C >0$ such that for all $n \leq N$,
%
%
\begin{eqnarray}
\label{eq:1arm} \pi_1(n,N) &\leq & (N/n)^{-\lambda},
\\
\label{eq:3arm} c (N/n)^{-2} & \leq & \pi_5(n,N) \leq C
(N/n)^{-2},
\\
\label{eq:5arm} c (N/n)^{-2} &\leq & \pi_3^{hp}(n,N)
\leq C (N/n)^{-2}.
\end{eqnarray}
\end{prop}

As a consequence of the above, we have the following estimates for the
probabilities of arm events of interest to us.
The proof is a simple application of Reimer's inequality \cite{Reimer}.
%

\begin{cor}\label{cor:la - 1}
There exist constants $c, \lambda>0$ such that, for all $n \leq N$,
\[
\pi_6(n,N) \leq c (N/n)^{-(2+\lambda)}\quad \mbox{and} \quad \pi
_4^{hp}(n,N) \leq c (N/n)^{-(2+\lambda)}.
\]
\end{cor}

Among the results of this section,
only the following corollary will be used explicitly in the rest of the paper.

%
\begin{cor}\label{cor:la}
There exist constants $c \geq1$ and $\lambda> 0$ so that, for all $n
\leq N$,
%
%
\begin{eqnarray}
\mathbb{P}_{p_c} \bigl(\mathcal{A}_6^*(n,N) \bigr) &\leq &
c(N/n)^{-(2+\lambda)}\quad \mbox{and}
\nonumber
\\[-8pt]
\\[-8pt]
\nonumber
\mathbb{P}_{p_c} \bigl(
\mathcal {A}_4^{hp *}(n,N) \bigr) &\leq &  c(N/n)^{-(2+\lambda)}.
\end{eqnarray}
\end{cor}
\begin{pf}
The statement above follows directly from Proposition~\ref
{prop:arms_defects} and Corollary~\ref{cor:la - 1}.
\end{pf}

%
\begin{figure}[b]

\includegraphics{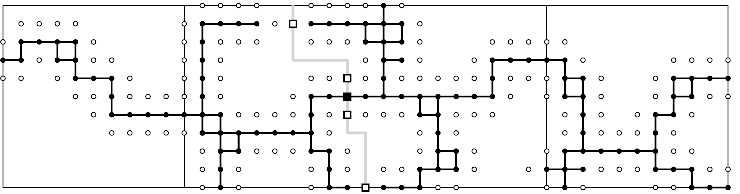}

\caption{The set $\chi$ in black, surrounded by $\partial\chi$.
The gray path $\gamma$ needs to use at least one passage point to cross
$\chi$ (black square).
Even if a site which is pivotal for $\{ \omega\in\mathcal{C}_h(S_n)
\}$ is enhanced,
$\gamma$~generally needs additional passage points to cross $\partial
\chi$
(see the empty squares).}
\label{fig:clusters}
\end{figure}

\section{Proof of Theorem~\texorpdfstring{\protect\ref{thm:cross}}{4}}\label{sec:cross}

\subsection{Plan of proof}\label{sec:sketch}

The proof of Theorem~\ref{thm:cross} is quite intricate;
we start with some notation and a brief description of the strategy.

Fix some $\delta> 0$ and $n \in\mathbb{N}$.
Consider a pair of configurations $\omega, \sigma$, and 
recall the definition of $\chi$, $\widetilde\omega$ and $\widetilde
\omega^\sigma$
from the lines above (\ref{def:omega_tilde}), (\ref{def:omega_tilde})
and (\ref{def:omega_tilde_delta}), respectively.
Call a point $x$ is called \emph{enhanced} if $\widetilde\omega(x) =
0$ but $\widetilde
\omega^\sigma(x) = 1$.
We will bound the probability $\mathbb{P}_{p_c,\delta} (\omega\in
\mathcal{C}_h(S_n) \mbox{
and } \widetilde\omega^\sigma\in\mathcal{C}_v(S_n))$,
which is obviously larger than $\mathbb{P}_{p_c,\delta} (\omega\in
\mathcal{C}_h(S_n) \mbox{
and } \widetilde\omega^\sigma\in\mathcal{C}_v(R_n))$.

If $\omega\in\mathcal{C}_h(S_n)$,
then $\chi$ contains a horizontal crossing of $S_n$.
If in addition there exists a $\widetilde\omega^\sigma$-open
vertical crossing of $S_n$,
then it must cross $\chi$, and hence it contains at least one enhanced
point; see Figure~\ref{fig:clusters}.

For $\omega\in\mathcal{C}_h(S_n)$ and $\sigma$ such that
$\widetilde\omega^\sigma\in\mathcal{C}_v(S_n)$, let $\gamma$ be
the left-most $\widetilde\omega^\sigma
$-open vertical crossing of $S_n$
containing the minimal number of enhanced points.
(We only take $\gamma$ to be left-most for it to be uniquely defined.)
Call the enhanced points of $\gamma$ \emph{passage points}, and let
$\mathscr{X}$
be the set of passage points.
If $\widetilde\omega^\sigma\notin\mathcal{C}_v(S_n)$ or $\omega
\notin\mathcal{C}_h(S_n)$, then let $\mathscr{X}
= \emptyset$.

Recall from the definition of $\widetilde\omega^\sigma$ that all
enhanced points are
contained in $R_n$.
Thus, under $\mathbb{P}_{p_c,\delta}$, $\mathscr{X}$ is a random set
of vertices of $R_n$,
nonempty when $\omega\in\mathcal{C}_h(S_n)$ and $\widetilde\omega
^\sigma\in\mathcal{C}_v(S_n)$.

We will prove (\ref{eq:cross}) by estimating the probability for
$\mathscr{X}$
to take specific values.
More precisely we will use the equality
%
%
\begin{equation}
\label{eq:probsum}
\mathbb{P}_{p_c,\delta} \bigl(\omega\in\mathcal{C}_h(S_n)
\mbox{ and } \widetilde \omega^\sigma\in\mathcal{C}_v(S_n)
\bigr) = \sum_{X \neq\emptyset} \mathbb{P}_{p_c,\delta}(
\mathscr{X}= X).
\end{equation}

The computation used to estimate $\mathbb{P}_{p_c,\delta}(\mathscr
{X}= X)$ is quite delicate.
Here are the main ideas; the actual proof is given in the following sections.

%
\begin{figure}[t]

\includegraphics{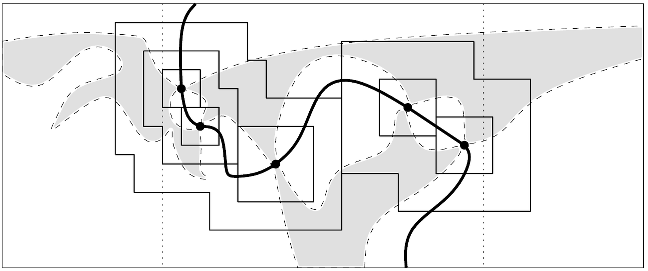}

\caption{The crossing $\gamma$ is drawn in bold and the passage
points are marked.
The set $\chi$ of sites open in $\omega$ but closed in $\widetilde
\omega$ is drawn
in gray. Its boundary is closed in $\omega$.
The blobs at the times of merger are outlined.
Observe the six-arm structure between the boundaries of the blobs.}
\label{fig:blobs}
\end{figure}

Fix a nonempty set of vertices $X$ with $|X| = k+1$, and let $\omega,
\sigma$ be configurations such that $\mathscr{X}= X$.
Since the passage points act as passages between the clusters of
$\widetilde\omega$,
they have, in $\omega$, a (local) six arm structure around them (see
Figure~\ref{fig:blobs}),
and we may control the probabilities of such configurations by $\pi_6$.

Imagine the following dynamics. Around each point $x \in X$ we grow a
ball at unit speed, $\Lambda_t(x) \dvtx t \geq 0$.
For a given time $t$, we call \emph{blobs} the connected components of
$\bigcup_x \Lambda_t(x)$.

For small times, the blobs are just balls centered at the points of $X$.
As time increases two blobs may \emph{merge} to create a bigger blob.
For a point $x\in X$ set $t(x)$ to be the first time of merger for the
blob containing $x$.
Thus $t(x) = \frac{1}2 \inf\{ \operatorname{dist}(x,y) \dvtx y \in X, y
\neq x\}$.
Then $\omega$ contains six arms from $x$ to $\partial\Lambda
_{t(x)}(x)$, an
event which has probability bounded by $\pi_6(t(x))$.
Moreover the regions $\Lambda_{t(x)}(x)$ for $x \in X$ are disjoint.
Finally, in order to be a passage point, $x$ has to be enhanced. This
happens with probability $\delta$, independently of $\omega$,
thus
%
%
\begin{equation}
\label{eq:bad_bound}
\mathbb{P}_{p_c,\delta}(\mathscr{X}= X) \leq\prod
_{x \in X} \pi_6 \bigl(t(x) \bigr) \delta.
\end{equation}

Unfortunately, this bound is not sufficient to obtain Theorem~\ref{thm:cross}.
If points are grouped in small bunches, then all values $t(x)$ are small,
and the right-hand side of (\ref{eq:bad_bound}) is not significantly
smaller than $\delta^{k+1}$.

In order to improve (\ref{eq:bad_bound}), we will also study the blobs
after their first mergers.
Consider a blob at the time of formation (e.g., by the merger
of two smaller blobs),
and the same blob at the first time it merges with another blob.
Let $B_1$ denote the blob at the initial time, and $B_2$ at the latter time.
Then we also observe six arms between $\partial B_1$ and $\partial B_2$.
This will add terms to the bound in (\ref{eq:bad_bound}), thus
improving it.

If we denote by $d_i$ the times of merger of blobs (counted with
multiplicity when more than two blobs merge at the same time),
then we obtain a bound on $\mathbb{P}_{p_c,\delta}(\mathscr{X}= X)$
as a function of
$d_1, \ldots, d_k$; see Proposition~\ref{prop:good_bound}.

In order to compute (\ref{eq:probsum}),
we also need to estimate the number of sets $X$ that yield a given set
of merger times $d_1,\ldots,d_k$.
This is done in Proposition~\ref{prop:X}.

In the above analysis we have omitted certain technical complications.
One is the influence of the boundary of $S_n$. As blobs expand,
they may touch the top and bottom of $S_n$, and special situations arise.
Another has to do with defects in arm events around passage points.

Before diving into the actual proof,
we mention that a simplified one-arm version of this argument already
appeared in \cite{Kes86IIC,BCKS}
for the study of the moments of the volume of the largest critical
percolation cluster in $\Lambda_n$.
It was shown there that the $k$th moment of this quantity is bounded
above by $k!(C n^2\pi_1(n))^k$ for a constant $C > 0$.
Contrary to the argument presented here,
in \cite{Kes86IIC,BCKS} blobs only need to be studied up to their first
time of merger,
and the resulting bound~(\ref{eq:bad_bound}) (with $\pi_6$ replaced by
$\pi_1$) suffices.
The fundamental reason for which~(\ref{eq:bad_bound}) suffices in that case
is that the one-arm exponent is smaller than $2$, hence $\sum_{k = 1}^n
k\pi_1(k) = O(n^2\pi_1(n))$.
The six-arm exponent, however, is larger than $2$, and the series $\sum_{k = 1}^{\infty} k\pi_6(k)$ converges,
thus requiring a more sophisticated analysis.

In \cite{K13}, the first author applies the refined counting arguments
presented here to the one-arm case
in order to derive an improved upper bound of $(Cn^2\pi_1(n/\sqrt{k}))^k$
for the $k$th moment of the volume of the largest critical cluster in
$\Lambda_{n}$.
These arguments lead to large deviation bounds for the volumes of large
critical percolation clusters.

\subsection{Two propositions}\label{sec:two_prop}

Fix some nonempty set $X \subset R_n$.
We associate to $X$ a~tree $\mathcal{T}= \mathcal{T}(X)$ as described below.
Although this is not important for our proof,
let us mention that $\mathcal{T}$ is a minimal spanning tree of $X$
and that the algorithm by which we construct it is Kruskal's
algorithm~\cite{K56}.

The vertices of $\mathcal{T}$ are the points of $X$, and the edges are added
successively as follows.

Let $\mathcal{T}_0$ be the graph with no edges and vertex-set $X$.
For $j \in\mathbb{N}$ define $\mathcal{T}_j$ by adding to $\mathcal
{T}_{j-1}$ a maximal set
of edges $(x,y)$ with $\operatorname{dist}(x,y) = j$,
which does not create cycles in $\mathcal{T}_{j}$.
Since $\operatorname{diam}(X) \leq4n$, $\mathcal{T}_j = \mathcal
{T}_{j+1}$ for $j\geq4n$, and we
define $\mathcal{T}= \mathcal{T}_{4n}$.
The graph $\mathcal{T}$ thus obtained
is indeed a tree: by construction it does not contain cycles and it is easy to check that it is connected.

Note that there is some ambiguity in the definition of $\mathcal{T}$
since there may be multiple choices
for the set of edges added to $\mathcal{T}_{j-1}$ to create $\mathcal{T}_j$.
To settle this, when multiple choices are available,
we choose the minimal one with respect to the lexicographical order of
$\mathbb{Z}^2 \times\mathbb{Z}^2$.
Let the root of $\mathcal{T}$ be the smallest element of $V(\mathcal
{T}) = X$ for the
lexicographical order of $\mathbb{Z}^2$.

Let $E(\mathcal{T})$ denote the edge-set of $\mathcal{T}$. Then $\#
E(\mathcal{T}) = k$.
For $e = (x,y) \in E(\mathcal{T})$, let $\mathrm{d}_{e} = \lfloor
\frac{1}2 \operatorname{dist}(x,y)
\rfloor+1$.
The multiset $\mathcal{D}(X) = [ \mathrm{d}_{e} \dvtx e \in\mathcal{T}]$
is called the set of
\emph{merger times} of $X$.

The terminology of merger times is inspired by the dynamics described
in Section~\ref{sec:sketch}.
Indeed, each edge $e$ of $\mathcal{T}$ corresponds to the merger of
two blobs
and $\mathrm{d}_{e}$ is their (approximate) time of merger.

%
\begin{prop}\label{prop:good_bound}
There exist constants $c, \lambda> 0$ such that,
for all $\delta>0$, $n \in\mathbb{N}$, and $X\subset R_n$ with $|X|
= k+1$, we have
%
%
\begin{equation}
\label{eq:good_bound}
\mathbb{P}_{p_c,\delta}(\mathscr{X}= X) \leq (c
\delta)^{k+1} n^{-(2 + \lambda)} \prod_{e \in
E(\mathcal{T})}
\mathrm{d}_{e}^{-(2 + \lambda)}.
\end{equation}
\end{prop}

Since the above offers a bound on $\mathbb{P}_{p_c,\delta}(\mathscr
{X}= X)$ as a
function of the set $\mathcal{D}(X)$,
it is natural to group the sum in (\ref{eq:probsum}) by the value of
$\mathcal{D}(X)$.

%
\begin{prop}\label{prop:X}
There exists a constant $K > 0$ such that,
for any given multiset of values $D = [d_1, \ldots, d_k]$,
the number of sets $X$ with $\mathcal{D}(X) = D$ is bounded as follows:
%
%
\begin{equation}
\label{eq:X} \# \bigl\{ X \subset R_n \dvtx \mathcal{D}(X)= D
\bigr\} \leq \mathcal{Q}(D) K^{k+1} n^2 \prod
_{i = 1}^{k} d_i,
\end{equation}
where $\mathcal{Q}(D)$ is the number of different ways of ordering $d_1,
\ldots,d_k$.
\end{prop}

Theorem~\ref{thm:cross} follows easily from the two propositions.

\begin{pf*}{Proof of Theorem~\ref{thm:cross} from Propositions~\ref{prop:good_bound} and~\ref{prop:X}}
Let $c, \lambda$ and $K$~be the constants provided by
Propositions~\ref
{prop:good_bound} and~\ref{prop:X}.
Choose $\delta> 0$ small enough to have $ c K \delta\sum_{d \geq1} d^{-(1
+ \lambda)} \leq1/2$.
It is essential here that $\lambda> 0$, so that the sum above converges.
Then, by Propositions~\ref{prop:good_bound} and~\ref{prop:X}, for $n
\in\mathbb{N}$,
\begin{eqnarray*}
&& \mathbb{P}_{p_c,\delta} \bigl(\omega\in\mathcal{C}_h(S_n)
\mbox { and } \widetilde \omega^\sigma\in\mathcal{C}_v(S_n)
\bigr)
\\
&& \qquad = \sum_{k \geq0} \sum
_{d_k \geq\cdots\geq d_1 \geq1} \mathop{\sum_{X \subseteq R_n}}_{\mathcal{D}(X) = [d_1, \ldots,d_k]}
\mathbb{P} _{p_c,\delta}(\mathscr{X}= X)
\\
&&\qquad \leq\sum_{k \geq0} (c K \delta)^{k+1}n^{-\lambda}
\sum_{d_k \geq
\cdots\geq d_1 \geq1} \mathcal{Q} \bigl([d_1,
\ldots,d_k] \bigr) \prod_{i=1}^k
d_i^{-(1 + \lambda)}
\\
&&\qquad  = c K \delta n^{-\lambda} \sum_{k \geq0} \biggl( c
K \delta\sum_{d \geq
1} d^{-(1 + \lambda)}
\biggr)^k
\\
&& \qquad \leq2 c K \delta n^{-\lambda}.
\end{eqnarray*}
%
\upqed\end{pf*}

\subsection{Proof of Proposition~\texorpdfstring{\protect\ref{prop:good_bound}}{14}}

The following lemma formalizes the fact that passage points have six
arms around them, possibly with a defect.

%
\begin{lemma}\label{lemma:arms}
Fix $n\in\mathbb{N}$:
\begin{enumerate}[(ii)]
\item[(i)]  Let $u \in S_n$ and $r \leq R$ such that $\Lambda_R(u) \subset S_n$.
If $\omega$ and $\sigma$ are configurations such that, when
$\widetilde{\omega}^\sigma$
is defined as in~(\ref{def:omega_tilde_delta}):
%
\begin{enumerate}[(a)] 
\item[(a)] $\Lambda_r(u)$ contains at least one
passage point,
\item[(b)]  $\Lambda_R(u) \setminus\Lambda_r(u)$
contains no
passage points.
\end{enumerate}
Then $\omega\in\mathcal{A}_6^*(u;r,R)$.
\item[(ii)]  Let $u \in\mathbb{Z}\times\{0,n\}$ and $r \leq R \leq n/2$. %
If $\omega,\sigma$ are configurations with the properties~\textup{(a)} and~\textup{(b)} above,
then $\omega\in\mathcal{A}_4^{hp *}(u;r,R)$.
\end{enumerate}
\end{lemma}

%
\begin{rem}\label{rem:notabuse}
In point (ii) above, when $u \in\mathbb{Z}\times\{n\}$,
we write $\mathcal{A}_4^{hp *}(u;r,R)$ for the event that there exist four
arms from
$\partial\Lambda_r(u)$ to $\partial\Lambda_R(u)$ in the half plane
below $\mathbb{R}
\times\{n\}$.
This is a slight abuse of notation that we will ask the reader to accept.
\end{rem}

\begin{pf*}{Proof of Lemma~\ref{lemma:arms}}
We start by giving a full proof of point (i);
we will then sketch the proof of (ii), marking the differences with
the previous point.

Let $u,r, R$ be as in (i).
For simplicity of notation we will write $A = \Lambda_r(u)$, $B =
\Lambda
_R(u)$ and $H= B\setminus A$.
Then $A \subset B\subset S_n$ and $A$ contains passage points, but $H$
does not.

Since $A$ contains passage points, $\gamma$ intersects $A$.
Thus we may find two disjoint sub-paths, $\gamma_1$ and $\gamma_2$,
of $\gamma$,
both contained in $H$, each connecting $\partial A$ to $\partial_iB$
and such that $\gamma$ contains at least one passage point between
$\gamma
_1$ and $\gamma_2$.
Let $\overline\gamma= \gamma_1 \cup\gamma_2$.
Then $\overline\gamma$ splits $H$ into two disjoint regions, $H^L$ and
$H^R$; see Figure~\ref{fig:six_arms}.

%
\begin{figure}

\includegraphics{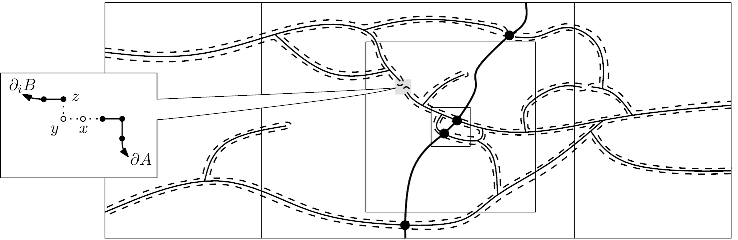}

\caption{Two concentric balls $A \subset B$ with two passage points in
$A$ but no passage points in $B \setminus A$.
Note the three arms on either side of $\gamma$.
The gray square marks the defect on one of the open arms.}
\label{fig:six_arms}
\end{figure}

Since $A$ contains passage points, there exists an $\omega$-open path
contained in $\chi$,
connecting $\partial A$ to the left-hand side of $S_n$.
This must contain a sub-path $\tau_1 \subset H$, connecting $\partial A$
to $\partial_iB$.
Since $\tau_1$ is $\omega$-open and $\widetilde\omega$-closed, it
can only
intersect $\gamma$ at passage points.
But, as part of $H$, $\tau_1$ does not contain passage points, thus is
fully contained in either $H^L$ or $H^R$.

Assume $\tau_1 \subset H^R$.
Then $\tau_1$ separates $H^R$ into two regions $H^{R+}$ and $H^{R-}$.
Let $\chi_R = \{ x \in H\dvtx  x  \stackrel{\omega, H}{\longleftrightarrow} \tau_1\}$ be
the $\omega
$-open cluster of $\tau_1$ in $H$.
The points of $\chi_R$ and those of $\partial\chi_R \cap H$ are closed
in $\widetilde\omega$ and are not passage points.
Thus they are not part of $\gamma$.
Hence $\partial\chi_R \cap H$ provides two paths $\tau_0$ and $\tau_2$
in $(\mathbb{Z}^2)^*$, contained in $H^{R-}$ and $H^{R+}$, respectively,
both closed in $\omega$ and connecting $\partial A$ to $\partial_iB$.

We have found up to now three arms $\tau_0, \tau_1, \tau_2$ in $H^R$,
with states closed, open and closed, respectively, in $\omega$.
It is natural to expect the same structure in $H^L$.
Some complications may arise though, hence the defect in the arm event.

Let $\chi_L = \{ x \in\chi\cap H^L \dvtx x \stackrel{\omega,H}{\longleftrightarrow}
\partial A \}$.
Also denote the ${}^{*}$-cluster of $\partial A$ in $H^L$ in the configuration
$\widetilde\omega$ by
$\Delta= \{x \in H^L \dvtx x \mathop{\stackrel{\widetilde\omega; H^L}{\longleftrightarrow}\!\!\!{}^*}   \partial A\}$.
Then $\chi_L \subset\Delta$.
First we claim that $\Delta$ must intersect $\partial_iB$.

Indeed, if it does not, consider the set $\partial(\Delta\cup A) \cap H^L$.
All sites of this set are $\widetilde\omega$-open.
Moreover, this set contains an $\widetilde\omega^\sigma$-open path
$\gamma'$ joining
$\gamma_1$ with $\gamma_2$.
Then $\gamma'$ contains no passage points, and this contradicts the
choice of $\gamma$
as having minimal number of passage points.

Let $\rho$ be a path in $\Delta$, connecting $\partial A$ to
$\partial_iB$.
Let $x$ be the last point of $\rho$ (when going from $\partial A$ toward
$\partial_iB$)
contained in $\chi_L \cup\partial\chi_L$.
If $x \in\partial_iB$, then there
exists a path $\tau_4$ connecting $\partial A$ to $\partial_iB$,
contained in
$\chi$ (hence $\omega$-open),
except possibly for its endpoint $x$.
Two additional $\omega$-closed arms $\tau_3$ and $\tau_5$ may be
found in
$H^L$ as previously done in $H^R$.

Suppose $x \notin\partial_iB$. Then $x\in\partial\chi_L$, and let
$y$ be the
next point visited by $\rho$.
Since $y \in\Delta$, there exists $z \sim y$ (or $z =y$) which is open
in $\omega$ but closed in $\widetilde\omega$.
In particular, $z$ is connected to the left side of $S_n$ by a $\omega
$-open path $\tau$ (part of $\chi$).
By choice of $x$, $z \notin\chi_L$, thus $\tau$ does not intersect $A$.
It does not intersect $\overline\gamma$ either, since the latter
contains no
passage points.
Thus $\tau$ contains a sub-path in $H^L$, linking $z$ to $\partial_iB$.

In conclusion there exists a path $\tau_4$ linking $\partial A$ to
$\partial_i
B$, contained in $H^L$
and open in $\omega$, with the possible exception of the sites $x$ and $y$.
With a possible modification of the configuration in $\Lambda_3(x)$,
two $\omega$-closed *-arms $\tau_3$ and $\tau_5$ may be found
by inspecting the boundary of the $\omega$-open cluster of $\tau_4$ in
$H^L$.

Since the arms $\tau_0, \tau_1, \tau_2$ are contained in $H^R$ and
$\tau_3, \tau_4, \tau_5$ are contained in $H^L$,
they are necessarily disjoint. This completes the proof of (i).

For (ii) consider $u \in[-3n,3n] \times\{0\}$ and $r \leq R \leq
n/2$ such that
$\Lambda_R(u) \setminus\Lambda_r(u)$ does not contain passage
points, but
$\Lambda_r(u)$ contains at least one.
In particular $\Lambda_R(u)$ intersects $R_n$, and since $R\leq n/2$,
$\Lambda_R(u) \cap\mathbb{H}\subset S_n$.

As before, write $A = \Lambda_r(u)$, $B = \Lambda_R(u)$ and $H=
(B\setminus A) \cap\mathbb{H}$.
In this case there exists a single sub-path $\overline\gamma$ of
$\gamma$
connecting $\partial A$ to $\partial_iB$.
Still $\overline\gamma$ splits $H$ into disjoint regions $H^L$ and $H^R$.

%
\begin{figure}[b]

\includegraphics{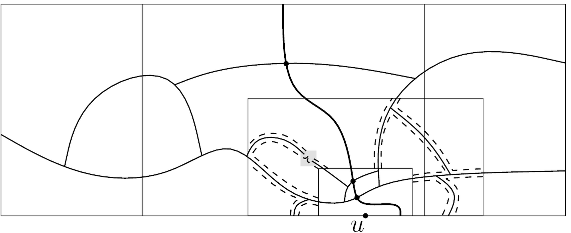}

\caption{The intersection of $\Lambda_r(u)$ with $\mathbb{H}$
contains passage points,
but $\Lambda_R(u) \setminus\Lambda_r(u)$ does not.
Then there are two open arms on either sides of $\gamma$ between
$\partial
\Lambda_r(u)$ and $\partial_i\Lambda_R(u)$.
Above each open arm (but not necessarily also below) there is a closed arm.}
\label{fig:six_hp_arms}
\end{figure}

We may proceed as before in defining $\tau_1$ and $\chi_R$.
The key difference with part~(i) is that only one part of the boundary
of $\chi_R$ is guaranteed to contain an $\omega$-closed arm.
Indeed, the part of the boundary above $\tau_1$ contains a path $\tau_2$,
contained in~$H^R$, closed in $\omega$, and connecting $\partial A$ to
$\partial_iB$.
The part below $\tau_1$, however, can intersect the bottom of $S_n$
very close to $\partial A$.
The same type of phenomenon takes place in~$H^L$.
In conclusion we obtain four arms in the half plane with one possible defect.
See also Figure~\ref{fig:six_hp_arms}.
\end{pf*}

We now turn to a consequence of Lemma~\ref{lemma:arms}
that will be used in the proof of Proposition~\ref{prop:X}.
To state it we need some additional notation.
Let $n\in\mathbb{N}$ and $A\subset B$ be two sets intersecting $R_n$.
Let $r = \lceil\operatorname{diam}(A)/2\rceil$.
Then there exists a vertex $u$ such that $A \subset\Lambda_{r}(u)$.
If several such vertices exist, let $u$ be the minimal one for the
lexicographical order of $\mathbb{Z}^2$.
If $\operatorname{dist}(u, \mathbb{R}\times\{0\}) \leq n/2$, let $v$
be the projection of
$u$ onto $\mathbb{R}\times\{0\}$.
Otherwise let $v$ be the projection of $u$ onto $\mathbb{R}\times\{n\}$.

We define the following additional quantities:
\begin{eqnarray*}
R &=& \sup \bigl\{s \in\mathbb{N}\dvtx \Lambda_s(u) \subset B \cap
S_n \bigr\} \vee r,
\\
r' &=& \inf \bigl\{s \in\mathbb{N} \dvtx \Lambda_R(u)
\subset \Lambda_s(v) \bigr\} \wedge n/2,
\\
R' &=& \bigl(\sup \bigl\{s \in\mathbb{N}\dvtx \Lambda_s(v)
\subset B \bigr\} \wedge n/2 \bigr)\vee r'.
\end{eqnarray*}
See Figure~\ref{fig:AB} for the meaning of $u,v,r,R,r'$ and $R'$.
Define the event
\[
\mathcal{E}(A,B) = \mathcal{A}_6^*(u;r,R) \cap\mathcal{A}_4^{hp *}
\bigl(v;r',R' \bigr).
\]
When $v \in\mathbb{R}\times\{n\}$, we use the notation $\mathcal{A}_4^{hp
*}(v;r',R')$ as described in Remark~\ref{rem:notabuse}.

%
\begin{figure}

\includegraphics{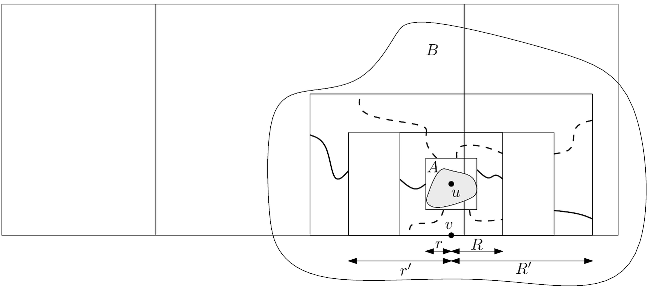}

\caption{Two sets $A \subset B$ intersecting $R_n$.
The six arms between $\partial\Lambda_r(u)$ and $\partial_i\Lambda_R(u)$
and the four arms in $\mathbb{H}$ between $\partial\Lambda_{r'}(u)$
and $\partial_i\Lambda_{R'}(u)$
ensure that $\mathcal{E}(A,B)$ occurs.}
\label{fig:AB}
\end{figure}

%
\begin{rem} \label{rem:notation_pi}
Henceforth we will write, for $n \leq N$, $\pi(n,N) = \pi(N/n) = c
(N/n)^{-(2+\lambda)}$,
where $c$ and $\lambda$ are given by Corollary~\ref{cor:la}.
This is to emphasize that the computations may be carried through with
different types of arm events with power-law behavior.
The quasi-multiplicativity property of probabilities of such events is
essential.
For $\pi$ it states that there exist constants $c_1,c_2 >0$ such that,
for all $n \leq m \leq N$,
%
%
\begin{eqnarray}
\label{eq:mult_pi} \pi(n,N) &\leq & \pi(n,m) \pi(m,N) \leq c_1
\pi(n,N),
\\
\label{eq:ext_pi} \pi(n,N) &\leq & c_2 \pi(n,2N).
\end{eqnarray}
\end{rem}


%
\begin{lemma}\label{lemma:Eprob}
\textup{(i)} Let $A \subset B$ be two sets of vertices of $\mathbb{Z}^2$.
If $\omega, \sigma$ are such that $A$ contains at least one passage point
and $B \setminus A$ contains none,
then $\omega\in\mathcal{E}(A,B)$.

\textup{(ii)} There exists a constant $c >0$ such that,
for all $n\in\mathbb{N}$ and all sets $A \subset B$ intersecting
$R_n$ with
$\operatorname{diam}(B) \leq6n$,
%
%
\begin{equation}
\label{eq:Eprob} \mathbb{P}_{p_c} \bigl(\mathcal{E}(A,B) \bigr) \leq c
\pi \bigl(\operatorname{diam}(A), \operatorname{diam}(A) + \operatorname{dist}
\bigl(A, B^c \bigr) \bigr),
\end{equation}
with $\pi(\cdot,\cdot)$ as in Remark~\ref{rem:notation_pi}.
\end{lemma}


\begin{pf}
(i) Let $A \subset B$ and $\omega, \sigma$ be as in the lemma.
With the notation in the definition of $\mathcal{E}(A,B)$, if $r < R$, then
$A \subset\Lambda_r(u) \subset\Lambda_R(u) \subset B \cap S_n$.
By Lemma~\ref{lemma:arms}(i), $\omega\in\mathcal{A}_6^*(u;r,R)$.
If $r = R$, then $\mathcal{A}_6^*(u;r,R)$ is trivial.

As in the previous paragraph, if $r' = R'$, then $\mathcal{A}
_4^{hp*}(v;r',R')$ is trivial.
Suppose that $r' < R'$. Without loss of generality we may assume $v
\in\mathbb{Z}\times\{0\}$.
Then
\[
A \cap\mathbb{H}\subset\Lambda_{r'}(v) \cap\mathbb{H}\subset
\Lambda_{R'}(v) \cap\mathbb{H}\subset B \cap S_n.
\]
By Lemma~\ref{lemma:arms}(ii), $\omega\in\mathcal{A}_4^{hp*}(v;r',R')$.
In conclusion $\omega\in\mathcal{E}(A,B)$.

(ii) If $R \geq n/4$, we have $\operatorname{diam}(B) \leq6n \leq24 R$.
Then the first inequality of Corollary~\ref{cor:la},
the fact that $\mathcal{A}_6^*(u;r,R) \subset\mathcal{E}(A,B)$
and~(\ref{eq:mult_pi}) yield (\ref{eq:Eprob}) after some simple
arithmetic manipulations.
Thus we may restrict ourselves to $R < n/4$.

We distinguish two cases.
First consider that $\Lambda_{R+1}(u)$ intersects $B^c$.
Then $R = \operatorname{dist}(u, B^c) - 1$,
and the first inequality of Corollary~\ref{cor:la} yields (\ref
{eq:Eprob}) as above.

Suppose now that $\Lambda_{R+1}(u)$ does not intersect $B^c$.
Then $\Lambda_R(u)$ necessarily intersects {$\mathbb{R}\times\{0,n\}$}.
It follows that $r' = 2R <n/2$ and $v \in B$.
By considering the cases $R'<n/2$ and $R'\geq n/2$ separately, we find
\[
R' \geq\frac{1}{12}\operatorname{dist} \bigl(v,
B^c \bigr) \geq \frac{1}{12} \bigl( \operatorname{dist} \bigl(A,
B^c \bigr) - r' \bigr).
\]
The second inequality of Corollary~\ref{cor:la}, equations (\ref
{eq:mult_pi}) and (\ref{eq:ext_pi})
and the above imply
\[
\mathbb{P}_{p_c} \bigl(\mathcal{A}_4^{hp *}
\bigl(v;r',R' \bigr) \bigr)\leq c' \pi
\bigl(R, \operatorname{dist} \bigl(A, B^c \bigr) \bigr),
\]
for some $c' > 0$.
In addition, by the first inequality of Corollary~\ref{cor:la}, we have
\[
\mathbb{P}_{p_c} \bigl(\mathcal{A}_6^*(u;r,R) \bigr) \leq
c'' \pi \bigl(\operatorname{diam}(A), R \bigr),
\]
for some $c'' > 0$.
Finally note that $\mathcal{A}_6^*(u;r,R)$ and $\mathcal{A}_4^{hp *}(v;r',R')$
depend on disjoint regions of the plane,
hence
\begin{eqnarray*}
\mathbb{P}_{p_c} \bigl(\mathcal{E}(A,B) \bigr) &\leq & c'
c'' \pi \bigl(\operatorname{diam}(A), R \bigr) \pi
\bigl(R, \operatorname{dist} \bigl(A, B^c \bigr) \bigr)
\\
&\leq & c \pi \bigl(\operatorname{diam}(A), \operatorname{diam}(A) +
\operatorname{dist} \bigl(A, B^c \bigr) \bigr),
\end{eqnarray*}
where $c > 0$ is obtained using again equations (\ref{eq:mult_pi}) and
(\ref{eq:ext_pi}).
\end{pf}
%
%
%

Finally we are ready for the proof of Proposition~\ref{prop:good_bound}.

\begin{pf*}{Proof of Proposition~\ref{prop:good_bound}}
Fix some nonempty set $X \subset R_n$ with $\#X = k+1$.
Let $e_1, \ldots, e_{k}$ be an ordering of the edges of $\mathcal{T}$
such that
the sequence $\mathrm{d}_{e_i}$ is increasing.

For an edge $e_i$ of $\mathcal{T}$, let $C_i$ be the set of vertices
of $\mathcal{T}$
connected to $e_i$ via edges~$e_j$ with $j \leq i$.
Let $\mathcal{C}= \{ C_i \dvtx i =1, \ldots, k\}$ and $\overline\mathcal
{C}: = \mathcal{C}\cup\{
\{x\} \dvtx x \in X\}$.
Inclusion provides a natural partial order of the elements of
$\overline\mathcal{C}
$. 
The singletons are the lowest elements; the maximal element is $X$.

For each $i=1,\ldots, k$, $C_i$ is the union of two smaller disjoint
elements of $\overline\mathcal{C}$, which we will call the offspring
of $C_i$.
If we write $e_i = (x,y)$, the offspring of $C_i$ are the connected
components of $x$ and $y$, respectively,
in the graph with vertices~$X$ and edges $\{e_1, \ldots, e_{i-1}\}$.

Thus the elements of $\overline\mathcal{C}$ form a binary tree with
the singletons
of $X$ as leaves.
We will sometimes refer to $\overline\mathcal{C}$ itself as a tree.
In the vision given in Section~\ref{sec:sketch}, $\overline\mathcal
{C}$ is the
coalescence tree of the blobs
(at least when blobs merge only two at a time).
Indeed, at time $\mathrm{d}_{e_i}$ two blobs merge and form a larger one,
that contains the vertices of $C_i$.
The two offspring of $C_i$ correspond to the two merging blobs.
If more than two blobs merge at the same time, we split this into
sequential pairwise mergers.

For $U \in\mathcal{C}$ let
\begin{eqnarray*}
\mathrm{d}_{U} &=& \biggl\lfloor\frac{1}{2} \max \bigl\{
\operatorname{dist}(x,y) \dvtx x,y \in U \mbox{ and } (x,y) \in E(\mathcal{T})
\bigr\} \biggr\rfloor,
\\
\Delta_U &=& \biggl\lfloor\frac{1}{2} \operatorname{diam}(U)
\biggr\rfloor,
\\
\Lambda_r(U) &=& \bigcup_{u \in U}
\Lambda_{r}(u),
\end{eqnarray*}
for $r\geq0$. For $U = \{x\}$ a singleton, set $\mathrm{d}_{U} =
\Delta_U = 0$
and $\Lambda_r(U) = \Lambda_r(x)$.

Consider $\omega$ and $\sigma$ such that $\mathscr{X}= X$.
For $U\in\mathcal{C}$, let $V = V(U)$ and $W = W(U)$ denote its offspring.
The two regions $\Lambda_{\mathrm{d}_{U}}(V) \setminus\Lambda
_{\mathrm{d}_{V}}(V)$ and
$\Lambda_{\mathrm{d}_{U}}(W) \setminus\Lambda_{\mathrm{d}_{W}}(W)$
are disjoint and do
not contain passage points.
On the other hand both $\Lambda_{\mathrm{d}_{V}}(V)$ and $\Lambda
_{\mathrm{d}_{W}}(W)$
contain passage points.
Thus, by Lemma~\ref{lemma:Eprob}(i), the event
\[
\mathcal{E}_U = \mathcal{E} \bigl( \Lambda_{\mathrm{d}_{V}}(V),
\Lambda_{\mathrm{d}_{U}}(V) \bigr) \cap\mathcal{E} \bigl( \Lambda
_{\mathrm{d}_{W}}(W), \Lambda_{\mathrm{d}_{U}}(W) \bigr),
\]
must occur in $\omega$.
By Lemma~\ref{lemma:Eprob}(ii) there exists some constant $c >0$ such that
%
%
\begin{equation}
\label{eq:Ebound} \mathbb{P}_{p} ( \mathcal{E}_U ) \leq c
\pi( \Delta_V + \mathrm{d}_{V}, \Delta_V +
\mathrm{d}_{U} )\pi( \Delta_W + \mathrm{d}_{W},
\Delta_W + \mathrm{d}_{U} ).
\end{equation}
To extend the definition of $\mathcal{E}_U$ to $U \in\overline
\mathcal{C}$, define it as
the full event (i.e., equal to $\Omega$) when $U$ is a singleton.

Since there are no passage points outside of $\Lambda_{\mathrm{d}_{X}}(X)$,
we also have $\omega\in\mathcal{E}_{\mathrm{out}} := \mathcal{E}( \Lambda
_{\mathrm{d}_{X}}(X), \Lambda
_{\mathrm{d}_{X}\vee n}(X))$.
Finally all passage points need to be enhanced, hence $\sigma(x) = 1$ for
all $x \in X$.
Thus
\[
\{\mathscr{X}= X\} \subset \biggl( \bigcap_{U \in\mathcal{C}} \{
\omega\in \mathcal{E}_U \} \biggr) \cap\{ \omega\in
\mathcal{E}_{\mathrm{out}}\} \cap \biggl( \bigcap_{x \in X}
\bigl\{ \sigma(x) = 1 \bigr\} \biggr).
\]
Note that the events $ \mathcal{E}_U \dvtx U \in\mathcal{C}$ and the
event $\mathcal{E}_{\mathrm{out}}$
are defined on disjoint parts of the plane. Hence, by (\ref{eq:Ebound}),
%
%
\begin{eqnarray}
\mathbb{P}_{p_c,\delta}(\mathscr{X}= X) &\leq & \biggl( \prod
_{U \in
\mathcal{C}} \mathbb{P}_{p_c,\delta}( \omega\in
\mathcal{E}_U) \biggr) \times \mathbb{P}_{p_c,\delta}( \omega\in
\mathcal{E}_{\mathrm{out}})\nonumber\\
&&{}\times \biggl( \prod_{x \in X}
\mathbb{P}_{p_c,\delta} \bigl( \sigma(x) = 1 \bigr) \biggr)
\nonumber
\\[-8pt]
\label{eq:good_bound2}\\[-8pt]
\nonumber
&\leq & (c \delta)^{k+1} \pi( \Delta_X +
\mathrm{d}_{X}, \Delta_X + \mathrm{d}_{X} +n
)
\\
&&{}\times\prod_{U \in\mathcal{C}} \pi( \Delta_V +
\mathrm{d}_{V}, \Delta_V + \mathrm{d}_{U} )
\pi( \Delta_W + \mathrm{d}_{W}, \Delta_W +
\mathrm{d}_{U}).
\nonumber
\end{eqnarray}

In estimating the product above, we will use an induction on the
binary tree $\overline\mathcal{C}$.
For $Y \in\mathcal{C}$ let
\[
\Phi(Y) = \prod_{U \in\mathcal{C}; U \subseteq Y} \pi( \Delta_V +
\mathrm{d}_{V}, \Delta_V + \mathrm{d}_{U} )
\pi( \Delta_W + \mathrm{d}_{W}, \Delta_W +
\mathrm{d}_{U}),
\]
and set $\Phi(Y) = 1$ when $Y$ is a singleton.

Let us prove by induction on the tree that there exists $c_0 > 0$ such
that, for all $Y \in\overline\mathcal{C}$,
%
%
\begin{equation}
\label{eq:rec} \Phi(Y) \leq\pi(\Delta_Y+ \mathrm{d}_{Y})
\prod_{U \in\mathcal{C}; U \subseteq Y} c_0 \pi(
\mathrm{d}_{U}).
\end{equation}
When $Y$ is a leaf of $\mathcal{C}$, that is, a singleton of $X$, then
$\Phi
(Y) = 1$,
and (\ref{eq:rec}) is trivially true for any $c_0 \geq1$.
Assume $Y$ is an element of $\mathcal{C}$ with offspring $Z_1,Z_2$.
We have
\[
\Delta_Y \leq\Delta_{Z_1} + \Delta_{Z_2} +
\mathrm{d}_{Y}.
\]
Thus for at least one $i \in\{1,2\}$,
$\Delta_Y + \mathrm{d}_{Y} \leq2 ( \Delta_{Z_i} + \mathrm{d}_{Y} )$.
Assume it is the case for $i = 1$.
Then
%
%
\begin{eqnarray}
\Phi(Y) & =& \Phi(Z_1)\Phi(Z_2) \pi(
\Delta_{Z_1} + \mathrm{d}_{Z_1}, \Delta_{Z_1} +
\mathrm{d}_{Y} )\pi( \Delta_{Z_2} + \mathrm{d}_{Z_2},
\Delta_{Z_2} + \mathrm{d}_{Y} )
\nonumber
\\
& \leq & \pi(\Delta_{Z_1} + \mathrm{d}_{Z_1}) \pi(
\Delta_{Z_1} + \mathrm{d}_{Z_1}, \Delta_{Z_1} +
\mathrm{d}_{Y} )
\nonumber
\\
&& {}\times\pi(\Delta_{Z_2} + \mathrm{d}_{Z_2}) \pi(
\Delta_{Z_2} + \mathrm{d}_{Z_2}, \Delta_{Z_2} +
\mathrm{d}_{Y} ) \prod_{U \in\mathcal{C}; U \subsetneq Y}
c_0 \pi(\mathrm{d}_{U})
\nonumber
\\
\label{eq:use_mult}
&\leq & c_1^2 \pi(\Delta_{Z_1} +
\mathrm{d}_{Y}) \pi(\Delta_{Z_2} + \mathrm{d}_{Y})
\prod_{U \in\mathcal{C}; U \subsetneq Y} c_0 \pi(
\mathrm{d}_{U})
\\
\label{eq:use_ineq}
&\leq & \frac{c_1^2 c_2}{c_0} \pi(\Delta_Y + \mathrm{d}_{Y})
\prod_{U \in\mathcal{C};
U \subseteq Y} c_0 \pi(
\mathrm{d}_{U}).
\end{eqnarray}
In (\ref{eq:use_mult}) we have used the quasi-multiplicativity property
of $\pi$ (\ref{eq:mult_pi}), hence the constant $c_1$.
In (\ref{eq:use_ineq}) we have used that $\pi(\Delta_{Z_1} + \mathrm{d}_{Y})
\leq c_2 \pi(2(\Delta_{Z_1} + \mathrm{d}_{Y} )) \leq c_2 \pi(\Delta
_Y + \mathrm{d}_{Y})$,
and $\pi(\Delta_{Z_2} + \mathrm{d}_{Y}) \leq\pi(\mathrm{d}_{Y})$.
The constant $c_2$
is given by (\ref{eq:ext_pi}).
In conclusion, the recurrence holds, provided that $c_0 \geq c_1^{2} c_2$.

Let us get back to bound (\ref{eq:good_bound2}). Using (\ref{eq:rec}),
we have
\begin{eqnarray*}
\mathbb{P}_{p_c,\delta}(\mathscr{X}= X) &\leq & (c \delta)^{k+1} \pi
( \Delta_X + \mathrm{d}_{X}, \Delta_X +
\mathrm{d}_{X} + n ) \pi(\Delta _X+ \mathrm{d}_{X})
\prod_{U \in\mathcal{C}} c_0 \pi(
\mathrm{d}_{U})
\nonumber
\\
&\leq & c_1 (c c_0 \delta)^{k+1} \pi(n) \prod
_{U \in\mathcal{C}} \pi(\mathrm{d}_{U}) \qquad
\mbox{by (\ref{eq:mult_pi})}.
\end{eqnarray*}
This proves Proposition~\ref{prop:good_bound}.
\end{pf*}

\subsection{Proof of Proposition~\texorpdfstring{\protect\ref{prop:X}}{15}}

We begin with a lemma.
The number of rooted trees with $n$ vertices is less than that of
rooted plane trees with $n$ vertices
(since these are rooted trees along with an ordering of the offspring
of each vertex).
Since the latter is well known to be the $n$th Catalan number (see, e.g., Theorem~3.2 of~\cite{Drmota2009}),
we find the following.
%

\begin{lemma}\label{lemma:catalan}
The number of rooted trees on $n$ vertices is less than
\[
c_n = \frac{1}{n+1} \pmatrix{2n
\cr
n} < 4^n,
\]
where $c_n$ is the $n$th Catalan number.
\end{lemma}

We turn to the proof of Proposition~\ref{prop:X}.

\begin{pf*}{Proof of Proposition~\ref{prop:X}}
Fix $n \in\mathbb{N}$, $k\geq0$, and let $D$ be a multiset of $k$ not
necessarily distinct natural numbers.

Consider a rooted tree $T$ with $k$ edges.
Let $v_0$ denote the root of $T$, and let $v_0,\ldots, v_k$
denote a fixed depth-first ordering of the vertices of $T$ when we
start at~$v_0$.
For $i \geq1$, let $e_i$ be the edge linking $v_i$ to $\{v_0,\ldots,
v_{i-1}\}$.
In addition, associate to each edge $e_i$ a number $d_i$ such that $[
d_1, \ldots, d_k ] = D$.
Thus $T$ is a rooted tree with decorated edges.

Let us bound the number of sets $X \subset R_n$ for which $\mathcal
{T}(X)$ is
isomorphic to $T$ in the sense of rooted trees with decorated edges.
[The decorations of $E(\mathcal{T}(X))$ are the merger times $\mathrm{d}_e$ defined
in the beginning of Section~\ref{sec:two_prop}.]
We will do this by placing the points of $X$ sequentially in $R_n$,
and counting at every stage the number of possibilities.

Since $X \subset R_n$, there are at most $4 n^2$ choices for the
position of $v_0$, which we denote by $x_0$.
Once $x_0$ is fixed, there are at most $8 \mathrm{d}_1$ choices for $x_1$,
the position of~$v_1$.
We continue in this fashion.
For every choice of $x_0, \ldots, x_{i-1}$, there are at most $8
\mathrm{d}
_{i}$ choices for $x_i$, the position of $v_i$.
In conclusion there are at most $4n^2 \prod_{i=1}^k 8 d_i$ sets of
points $X \subset R_n$ with $\mathcal{T}(X)$
isomorphic to $T$ in the sense of rooted decorated trees.

To compute the number of sets $X \subset R_n$ with $\mathcal{D}(X) = D$,
we need to consider all possible values of $T$ and all the different
ways of assigning the decorations $d_i$ to its edges.
By Lemma~\ref{lemma:catalan} there are at most $4^k$ choices for $T$.
The number of ways to assign the decorations is obviously bounded by
$\mathcal{Q}(D)$.
Proposition~\ref{prop:X} follows with~$K = 8\cdot4$.
\end{pf*}

\section*{Acknowledgments}

The first two authors thank IMPA for its hospitality during August
2013, when this project took place.
The project was initiated while the first author was working at Centrum
Wiskunde \& Informatica (CWI), Amsterdam.
He is grateful to CWI for the hospitality and to NWO for the support in
that period.
He thanks ENS for its hospitality during a week long visit and G\'
{a}bor Pete and Artem Sapozhnikov for
fruitful discussions.

All three authors are grateful to Rob van den Berg for introducing them
to the model and
for comments on the earlier versions of the paper.
They also thank Hugo Duminil-Copin for various comments and discussions.

%

%




%

\printaddresses
\end{document}